\documentclass{amsproc}

\usepackage{amsmath,amsthm,amssymb,tikz-cd,todonotes,mathrsfs,stmaryrd}
\usetikzlibrary{decorations.pathreplacing,angles,quotes}

\theoremstyle{plain}
\newtheorem{thm}[subsection]{Theorem}
\newtheorem{prop}[subsection]{Proposition}
\newtheorem{lem}[subsection]{Lemma}
\newtheorem{clm}[subsubsection]{Claim}

\theoremstyle{definition}
\newtheorem{defn}[subsection]{Definition}

\theoremstyle{remark}

\newtheorem{rmk}[subsection]{Remark}
\newtheorem{case}{Case}

\newcommand{\nc}{\newcommand}
\nc{\dmo}{\DeclareMathOperator}

\nc{\B}[1]{\mathbb{#1}}
\nc{\C}[1]{\mathcal{#1}}
\nc{\Sc}[1]{\mathscr{#1}}
\dmo{\sbld}{\preceq_s}
\dmo{\bld}{\preceq}
\dmo{\bldneq}{\prec}
\dmo{\sbldneq}{\prec_s}
\nc{\lb}{\llbracket}
\nc{\rb}{\rrbracket}
\dmo{\Id}{Id}

\newcommand{\R}{\mathbb{R}}

\newcommand{\Z}{\mathbb{Z}}

\newcommand{\N}{\mathbb{N}}

\renewcommand{\mod}[1]{{\ifmmode\text{\rm\ (mod~$#1$)}\else\discretionary{}{}{\hbox{ }}\rm(mod~$#1$)\fi}}

\begin{document}
\title[Partially Bounded Transformations]{Partially Bounded Transformations have trivial centralizers}
\date{\today}
\author[Gaebler]{Johann Gaebler}
\address[Johann Gaebler]{Harvard University, Cambridge, MA 02138, USA}
\email{jgaebler@college.harvard.edu}

\author[Kastner]{Alexander Kastner}
\address[Alexander Kastner ]{ \\
     Williams College \\ Williamstown, MA 01267, USA}
\email{ ask2@williams.edu}

\author[Silva]{Cesar E. Silva}
\address[Cesar E. Silva]{Department of Mathematics\\
     Williams College \\ Williamstown, MA 01267, USA}
\email{csilva@williams.edu}

\author[Xu]{Xiaoyu Xu}
\address[Xiaoyu Xu]{Princeton University\\ Princeton, NJ 08544, USA}
\email{xiaoyux@princeton.edu}

\author[Zhou]{Zirui Zhou}
\address[Zirui Zhou]{University of California, Berkeley, Berkeley, CA 94720, USA}
\email{zirui\_zhou@berkeley.edu}


\maketitle

\date{\today}

\begin{abstract}
	We prove that for infinite rank-one transformations satisfying a  property called ``partial boundedness,'' the only commuting transformations are powers of the original transformation. This shows that a large class of infinite measure-preserving rank-one transformations with bounded cuts have trivial centralizers. We also characterize when partially bounded transformations are isomorphic to their inverse. \end{abstract}

\section{Introduction}
Given an ergodic  measure-preserving transformation there has been interest in understanding the set of transformations that commute with it. For example, it is well known that if $T$ is an irrational rotation and $T\circ S=S\circ T$ a.e., then $S$ must also be a rotation. In \cite{Or72}, Ornstein constructed mixing finite measure-preserving transformations that commute only with their powers; these were the first examples of  mixing transformations with no roots. Later, del Junco \cite{dJ79} proved that the well known Chac\'on transformation commutes only with its powers. All these examples are rank-one transformations, a  rich class of transformations for the construction of examples and counterexamples. In \cite{Ki86}, King showed that if $T$ is a rank-one finite measure-preserving transformation and $S$ commutes with $T$, then $S$ must be a limit of powers of $T$ in the weak topology. This theorem, known as the Weak Closure Theorem, has many interesting consequences, and has also been shown for finite measure-preserving rank-one flows \cite{JRR12}. It is known not to hold for rank-one finite measure-preserving  $\mathbb Z^2$ actions \cite{DoKw02}, \cite{DoSe09}, or rank-\(k\) actions for \(k>1\) \cite{KwLa02}. It  remains open for infinite measure-preserving rank-one transformations.

In the case of infinite measure-preserving, or nonsingular, transformations some results are known. In \cite{AaNa87}, Aaronson and Nakada constructed nonsingular group rotations with an equivalent ergodic infinite invariant measure that  commute only with their powers (having no measure-preserving factors though they have non-\(\sigma\)-finite factors); these examples have non-ergodic Cartesian square and  include the Hajian-Kakutani example \cite{HaKa70}. In \cite{RuSi89}, Rudolph and Silva constructed nonsingular transformations satisfying the property of rational minimal self-joinings, yielding nonsingular transformations of all Krieger types, and in particular rank-one infinite measure-preserving transformations commuting only with their powers (and having no nontrivial factors); these examples have ergodic Cartesian square. More recently,  Ryzhikov and Thouvenot \cite{RyTh15} have shown that infinite rank-one transformations that are Koopman mixing (or zero type) commute only with their powers. In \cite{JRR}, Janvresse, de la Rue, and Roy, prove that the infinite Chac\'on transformation of \cite{AdFrSi97} commutes only with its powers; this transformation is known to have infinite ergodic index \cite{AdFrSi97}.

In this paper we define the class of partially bounded rank-one transformations and show that each transformation in this class commutes only with its powers (in this case we say that the transformation has \emph{trivial centralizer}). Partial boundedness can only occur in infinite measure; it defines a broad class that includes transformations such as the Hajian-Kakutani transformation and the infinite Chac\'on transformation. Our methods are different from those in \cite{AaNa87}, \cite{RuSi89}, \cite{RyTh15} which use joining arguments; further it can be shown that partially bounded transformations are partially rigid and hence not Koopman mixing. We instead work with symbolic properties of the transformation as started in del Junco \cite{dJ79}. In fact we extend to infinite measure the methods of Gao and Hill \cite{GaHi16b} (see also Gao-Hill \cite{GaHi16}), who showed that canonically bounded finite measure-preserving rank-one transformations have trivial centralizers.

We also characterize the partially bounded transformations that are isomorphic to their inverse. As is well known, the isomorphism question for classes of transformations in ergodic theory has a long history
(see \cite{FoRuWe11}). Foreman, Rudolph and Weiss in \cite{FoRuWe11} give an argument that the isomorphism
problem for rank-one transformations is tractable. More recently,   in \cite{Hi16}, Hill characterizes the canonically bounded rank-one transformations that are isomorphic to their inverses. In the last section we extend the methods of \cite{Hi16} to partially bounded (infinite measure-preserving) transformations  to consider the 
inverse isomorphism problem.

Section~\ref{S:rank1} recalls the definition of rank-one transformations and symbolic systems using rank-one words. We introduce  partially bounded transformations and in Lemma~\ref{lemParBouUni} we describe the \(T\)--\(P\) names of these transformations.  In Section~\ref{S:ex} we show how a large class of examples can be rewritten to satisfy the partially bounded definition. 
In Proposition~\ref{propRigUncCen} we note that as in the finite measure-preserving case, rigid transformation have uncountable centralizer.
Section~\ref{S:prelim} develops in more detail properties of symbolic representations. In Theorem~\ref{thmMain} we show that partially bounded rank-one transformations have trivial centralizers. Rigid,  rank-one,  infinite measure-preserving transformations  are generic by \cite{AgSi02,BSSSW15}, and by  Proposition~\ref{propRigUncCen} they have uncountable centralizer. Section~\ref{S:isoinv}  gives a theorem that characterizes when a partially bounded transformation is isomorphic to its inverse. 

\textbf{Acknowledgments:} This paper is based on research done in the ergodic theory group of the 2016 SMALL   research project at Williams College. Support for the project was provided by National Science Foundation grant DMS-1347804, the  Science Center of Williams College, and the Williams College Finnerty Fund. We would like to thank  Madeleine Elyze,  Juan Ortiz Rhoton, and Vadim Semenov, the  other members of the SMALL 2016 ergodic theory group, for useful discussions and continuing support.  

\section{Rank-One Systems and \(T\)--\(P\) Names}\label{S:rank1}

\subsection{Notations}

In this paper, we shall consider measure-preserving transformations of an infinite, \(\sigma\)-finite  nonatomic Lebesgue measure space \((X,\Sc B,\lambda)\), e.g., \(\B R\) with Lebesgue measure. When we say that a transformation is invertible, that two sets are equal, and so on, there is always a tacit ``a.e.''

We shall make use of the following notations. We shall principally be concerned with words, i.e., finite or infinite sequences of some finite alphabet \(A\). We denote by \(\lb n,m\rb\) the finite subsequence of integers consisting of \(k\) such that \(n\leq k< m\), and, if \(x\) is some \(A\)-word, \(x\lb n,m\rb\) denotes the finite subword of length \(m-n\), indexed so that \(x\lb n,m\rb(0)=x(n)\), \(x\lb n,m\rb(1)=x(n+1)\), and so on. If \(u\) is a finite word, then \(|u|\) denotes its length. We say that a word \(v\) \text{occurs} in a word \(u\) at \(i\) when \(u\lb i,i+|v|\rb=v\). If \(u\) and \(w\) are words, then by \(uw\) we mean the concatenation of \(u\) and \(w\). By \(w^n\), we mean
	\[w^n = \overbrace{ww\cdots w}^{n\text{ times}}.\]

Let us restrict our attention to the alphabet \(\{0,1\}\), and to \(\C F\), the set of all finite words beginning and ending with 0. For \(u,w\in\C F\), we say that \(u\) \emph{builds} the word \(w\) (which we shall also write \(u\bld w\)) if 
	\[w=u1^{a_1}u1^{a_2}\cdots1^{a_r}u.\]

We may extend this relation to infinite (respectively, bi-infinite) words in the following way: we say that \(w\) builds \(W\in\{0,1\}^\B N\) if there are integers \(\{a_n\}_{n\in\B N}\) (respectively, \(n\in\B Z\)) such that
	\[W=w1^{a_0}w1^{a_1}\cdots\quad (\text{respectively,}\ W=\cdots 1^{a_{-1}}w1^{a_0}w\cdots).\]
As in the finite case, if the \(a_n\) are all equal, we write that \(w\) simply builds \(W\). Other notions, such as occurrence, extend to (respectively,~bi-) infinite words in the same way.

\subsection{Rank-One Transformations}

Rank-one transformations are an important class of measure-preserving transformations. They are ergodic, invertible, generic in the group of invertible measure-preserving transformations (for the infinite measure-preserving case see \cite{BSSSW15}), and an important source of examples and counterexamples in ergodic theory. We shall define them by the method of \emph{cutting and stacking}, although this definition is equivalent to a variety of other definitions; see \cite{Fe97} for finite measure and \cite{BSSSW15} for infinite measure.

The construction is inductive and proceeds by defining a sequence of columns. A {\it column} $C=\{I_0,\ldots, I_{h-1}\}$ consists of a finite sequence of disjoint intervals of the same length,  called {\it levels}, that we order from 0 to $h-1$, where $h$ is an integer called the {\it height} of the column. A column $C$ defines a partial  transformation by sending level $I_j $ to level $I_{j+1}, j\in\{0,\ldots,h-2\}$, by the unique translation between the two intervals, so we can write $I_{j+1}=T(I_j)$, and call $I_0$ the {\it base} of the column. We now define the  rank-one transformation corresponding to the sequence $(r_n)$ of numbers of {\it cuts}, $r_n\in\{2,3,\ldots\}$, and the sequence $(s_n)$ of $r_n$-tuples of numbers of {\it spacers}, $s_n(i)\in\{0,1,2,\ldots\}$ for $0 \leq i \leq r_n-1$. The base step is a column $C_0=\{B_0\}$ consisting of a single interval $B_0$ (which in the case of infinite measure may always be assumed  to be the unit interval $[0,1)$). Given column  \(C_n=\{B_n, T(B_n),\ldots, T^{h_n-1}(B_n)\}\), to obtain $C_{n+1}$ cut each level of $C_n$ into $r_n$ subintervals of equal length, ordered from left to right, place $s_n(i)$ new subintervals (called {\it spacers}) above the $i$th topmost subinterval of $C_n$, and stack right above left to form column $C_{n+1}$ of height 
	\begin{equation}
	\label{E:height}
		h_{n+1}= r_n h_n+\sum_{i=0}^{r_n-1}s_n(i).
	\end{equation}


To be more precise with  the inductive step, cut $B_n$ into $r_n$ equal-length subintervals denoted $B_{n,i}$ for 
$i=0,\ldots r_n-1$. For each $i=0,\ldots, r_n-1$ choose new intervals $S_{n,i,k}$
  (where $k$ ranges from 0 to $s_n(i)-1$) 
of the same length as $B_{n,i}$, and place them above $T^{h_n-1}(B_{n,i})$ to
form the new $i$th subcolumn
\[C_{n,i}=\{B_{n,i}, T(B_{n,i}),\ldots,T^{h_n-1}(B_{n,i}),S_{n,i,0},S_{n,i,1},\ldots,S_{n,i,s_n(i)-1}\}.\]
Extend $T$ so that is sends $T^{h_n-1}(B_{n,i})$ to $S_{n,i,0}$, each spacer to the one above it, and the top spacer, namely  $S_{n,i,s_n(i)-1}$, to the bottom of the next subcolumn, namely
$B_{n,i+1}$ if $i < r_n-1$ and it remains undefined when $i=r_n-1$ (the top spacer of the last subcolumn). (Note that when $i=r_n-1$ the spacer 
$S_{n,i,s_n(i)-1}$ becomes the top level of column $C_{n+1}$ and the transformation  will be defined on part of this at the next stage of the construction.)  This is the process of stacking the right over the left subcolumn and forms column $C_{n+1}$ with base $B_{n+1}=B_{n,0}$ 
and height $h_{n+1}$ as in Equation \eqref{E:height}.  We  note that every level in $C_n$ is a union of levels in $C_{n+1}$.
We let $X$ be the union of all the levels of all the columns. We note that $T$ is 
defined on all levels of $C_n$ except the top and $C_{n+1}$ extends the definition 
to part of the top level of $C_n$. As the length of the levels goes to 0, in the limit this defines 
a  transformation on $X$. It is measure-preserving as levels are sent to levels of the same length and it is invertible a.e. 
(by convention intervals are left-closed, right open, and if we delete the positive orbit of 0 the transformation will be invertible everywhere). 
It is   ergodic, as the levels can be made arbitrarily full of any measurable sets, see, e.g. \cite{si08}.

%
%
%
%
%

\subsection{Rank-One Words}

To the cutting and stacking construction of a rank-one system we  associate so-called \textit{rank-one words} in the following way. Let \(T\) be constructed using cutting parameter $(r_n)$ and stacking parameter $(s_n)$, where we assume that for all $n\geq 0$, $s_n(r_n-1)=0$, i.e., no spacers are placed on the final subcolumn at any stage of the construction; as we argue below it is not hard to see that this is not a restriction.  Then, we inductively define
	\begin{align*}
	w_0=0 \text{ and }
	w_{n+1}=w_n1^{s_n(1)}\cdots1^{s_n(r_n-2)}w_n.
	\end{align*}
It is immediately obvious that \(w_n\bld w_m\) for \(n\leq m\). We call \(w_n\) the \emph{\(n\)-th rank-one word associated to \(T\)}. There is a unique infinite word \(W\) such that \(w_n\bld W\) for all \(n\). The infinite word defined this way is said to be the \textit{infinite rank-one word} associated to \(T\). It is clear that one can read the cutting and spacer parameters from the rank-one words. We also remark that the same transformation may have different presentations in terms of its cutting and spacer parameters.


Now we show that the fact that we require that no spacers be placed on the final subcolumn at any stage of the construction does not restrict the definition of rank-one transformations; we can delay the addition of the spacers on the final subcolumn to subsequent steps in the construction. That is, we can set
    \begin{itemize}
    \item \(s^\prime_0(i) =s_0(i)\) for \(i=0,\ldots,r_{0}-2\), and \(s'_0(r_0-1)=0\),
    \item for \(n>0\), \(s'_n(i) =\sum_{k=0}^{n-1} s_k(r_k-1)+s_n(i)\) for \(i=0,\ldots,r_n-2\), and  \(s^\prime_n(r_n-1)=0\).
    \end{itemize}
The rank-one transformation defined by cutting and spacer parameters $(r_n)$ and $(s'_n)$ will be isomorphic to the transformation defined by $(r_n)$ and $(s_n)$. 

As an illustration, the construction of the infinite Chac\'on transformation is usually in terms of the parameters \(r_n=3\), , \(s_n(0)=0\), \(s_n(1)=1\), \(s_n(2)=3 h_n+1\), for all $n\geq 0$ \cite{AdFrSi97}. In the standard construction, to obtain $C_1$ after $C_0$ we subdivide the interval in $C_0$ 
intro three subintervals, put one spacer above the middle subinterval and four spacers above the last subinterval. In the modified construction, subdivide the subinterval in $C_0$ 
into three subintervals, put a single spacer in the middle subinterval and no spacer above
the last subinterval; this gives a column $C_1^\prime $ of height 4. Next subdivide 
each level of $C_1^\prime $ into three subintervals and now put four spacers on top of 
the first subinterval, four plus one on top of the second subinterval and none on the last subinterval. (Note that the top level of $C_1^\prime $ is the same as the last subinterval of $C_0^\prime$.) This process defines
the following rank-one words:
	\begin{align*}
		&w_0	&&=0\\
		&w_1	=w_0 1^{s'_0(0)} w_0 1^{s'_0(1)} w_0	&&=0010\\
		&w_2	=w_1 1^{s'_1(0)} w_1 1^{s'_1(1)} w_1	&&=001011110010111110010
	\end{align*}
Note that in this case, \(|w_n|=h_n\), where \(h_n\) indicates the height parameter of the modified construction, as we have not added any spacers to the final subcolumn.

\subsection{\(T\)--\(P\) Names}

Each point \(x\in X\) also has an associated infinite \(T\)--\(P\) name for any finite partition \(P\) of \(X\). For our purposes, the desired partition of \(X\) is \(P=\{B_0,X\setminus B_0\}\).

\begin{defn}
The \(T\)--\(P\) {\it name}  of \(x\) is the bi-infinite sequence of 0s and 1s denoted by $\Theta(x)$ and defined for $i \in \mathbb{Z}$ by
	\[\Theta(x)(i) = 
	\begin{cases}
	  0   &\text{if } T^i(x)\in B_0,\\
	     1 & \text{if }T^i(x)\notin B_0. 
	     \end{cases}\]
     \end{defn}

We observe that if $T$ is any rank-one transformation and $P = \{B_0, X \setminus B_0\}$, then any rank-one word $w_n$ builds $\Theta(x)$ for a.e.~$x \in X$. Moreover there is no first or last occurrence of $w_n$ in $\Theta(x)$ for a.e.~\(x\in X\). We desire that the \(T\)--\(P\) name so-generated be unique for a.e.~\(x\in X\), i.e., that \(\Theta\) be injective. We prove this for the following subclass of rank-one transformations.

\begin{defn}
\label{defParBou}
A rank-one transformation \(T\) is said to be \emph{partially bounded} if it admits cutting and spacer parameters $(r_n)$ and \((s_n)\) with no spacers added to the final subcolumn at any stage, and satisfying the following property: There are integers \(\mathfrak R > 0\) and \(\mathfrak S > 0\) such that for all $n\geq N$, for some integer $N$, 
	\begin{enumerate}
		\item  \(r_n<  \mathfrak R\);
		\item \(|s_n(i)-s_n(j)|<\mathfrak S\),  for \(0\leq i,j<r_n-1\); 
		\item  \(s_n(i)\geq|w_n|\), for  \(0\leq i<r_n-1\).
	\end{enumerate}
\end{defn}

Note that by condition (3) all partially bounded transformations act on infinite measure spaces. We may and do assume without loss of generality that \(N=0\).

\begin{defn}
\label{defExp} It follows from the definition of \(T\)--\(P\)  names that if $T^i(x) \in B_n$, then the  \(T\)--\(P\)  name of $x$ has an occurrence of $w_n$ at $i$. We say that such an occurrence of \(w_n\) is \emph{expected}; otherwise, if \(T^i(x)\notin B_n\), it is called \emph{unexpected}.
\end{defn}

\begin{lem} \label{lemParBouUni}
Let $T$ be a partially bounded rank-one transformation, and consider the partition $P = \{B_0, X \setminus B_0\}$. Let $(w_n)$ denote the sequence of rank-one words corresponding to the partially bounded parameters. Then:
\begin{enumerate}
\item The map $\Theta$ that sends a point to its $T-P$ name is injective.
\item All occurrences of $w_n$ in $\Theta(x)$ are expected.
\end{enumerate}
\end{lem}
\begin{proof}
Statement (1) says exactly that the collection $\{T^n(B_0), T^n(X \setminus B_0)\}_{n\in\B Z}$ separates points. If $x$ and $y$ are distinct points, then there exists some column $C_m$ where $x$ and $y$ belong to different levels. Suppose without loss of generality that $x$ belongs to a lower level than $y$. Choose the smallest $n \geq 1$ such that $T^n(x)$ belongs to the top level of $C_m$. Then by condition (3) of partial boundedness, $T^n(x)$ will be in the spacers so that $T^n(x) \in [0,1)$ and $T^n(y) \not \in [0,1)$.

Statement (2) follows immediately from condition (3) in the definition of partial boundedness.
\end{proof}



We observe  that, given an infinite rank-one word \(W\) coming from the cutting and spacer parameters of the rank-one transformation \(T\), it is possible to generate a symbolic system \((S(W),\Sc M,\mu,\sigma)\) in the following way: let
	\[S(W) :=\left\{x\in\{0,1\}^\B Z:(\forall\ n,m)(\exists\ n',m')\big(x\lb n,m\rb=W\lb n',m'\rb\big)\right\},\]
  in other  words, every finite subsequence of \(x\) is a subsequence of \(W\). We set \(\sigma(x)(i)=x(i+1)\), i.e., \(\sigma\) is the shift operator. We define  the cylinder sets
	\[E_{w_n,i} :=\{x\in S(W):\text{ \(x\) has an expected occurrence of \(w_n\) at \(i\)}\}.\]
Then $E_{w_n,0}$ corresponds to the base level $B_n$ of the $n$-th column, and $E_{w_n,i}$ corresponds to $T^{-i}(B_n)$ (recall that every occurrence of $w_n$ is expected for partially bounded transformations).
To obtain a measure,  define $\mu(E_{w_0,0}) = 1$ and extend $\mu$ uniquely to a Borel $\sigma$-finite atomless shift-invariant measure on $X$; this is the same as the push-forward measure from the geometric construction (as we choose to always start with the unit interval).
%
%
Then \(\Theta\) will be an isomorphism between \(X\) and \(S(W)\) for all partially bounded transformations; we use this isomorphism to identify $X$ with the symbolic space and $T$ with the shift.

\section{Examples}\label{S:ex}

A wide collection of infinite rank-one transformations fall into the class of partially bounded transformations, including the infinite Chac\'on transformation \cite{AdFrSi97} and the infinite Hajian-Kakutani \cite{HaKa70} transformation (defined by \(r_n=2\), \(s_n(1)=0\), \(s_n(2)=2h_n+1\)). As outlined in Section 2.3, these transformations have an alternative presentation with no spacers on the final subcolumn. In fact, when using the cutting and stacking construction notation, all rank-one transformations that have bounded cuts \((r_n)\), and uniformly bounded spacers on the first \(r_n -1\) columns but having a very large number of spacers added on the last column can be shown to have a presentation as a partially bounded transformation. The formal statement goes as follows:

\begin{lem}
\label{lemExpGro}
A rank-one transformation with cutting parameters \((r_n)\) and spacer parameters \((s_n)\) satisfying for sufficiently large $n$,
	\begin{enumerate}
		\item \(r_n \leq \mathfrak R\),
		\item \(s_n(i) < \mathfrak S\) for \(0 \leq i \leq r_n-2\),
		\item\(s_n(r_n-1)\geq h_{n+1}/2\),
	\end{enumerate}
has an isomorphic presentation in terms of the parameters \(r_n\) and 
\[s'_n(i)=s_n(i)+\sum_{k=0}^{n-1}s_k(r_k-1)\text{ for }0\leq i\leq r_n-2\text{ and }s'_n(r_n-1)=0\] which is partially bounded.
\end{lem}


\begin{proof}
By assumption, for all sufficiently large $n$, we have $r_n \leq \mathfrak R$ and $|s'_n(i)-s'_n(j)|=|s_n(i)-s_n(j)|<\mathfrak S$. Let $(w_n)$ denote the sequence of rank-one words associated to the parameters $(r_n)$ and $(s'_n)$. Then for all $0 \leq i \leq r_n-2$,
	\[|w_{n+1}|\leq h_{n+1} - s_n(r_n-1) \leq s_n(r_n-1) \leq s'_{n+1}(i)\]
\end{proof}

We now study a second class of transformations.

\begin{defn}
\label{defCen}
For an invertible measure-preserving transformation \(T\) of \(X\), by \(C(T)\) we denote the \emph{centralizer} of \(T\), i.e. the set of all invertible measure-preserving transformations of \(X\) commuting with \(T\). If \(C(T)=\{T^n\}_{n\in\B Z}\), then \(T\) is said to have \emph{trivial centralizer}.
\end{defn}

\begin{defn}
\label{defRig}
Given \((X,\Sc B, \mu, T)\), a \(\sigma\)-finite Lebesgue measure space with  a measure-preserving transformation \(T\), if there is an increasing  sequence of integers \(n_k\) such that for all sets \(A\) of finite measure,
	\[\lim_{k\to\infty}\mu[T^{n_k}(A)\triangle A]=0\]
we say that \(T\) is {\it rigid}.
\end{defn}

The following proposition is well known in the finite measure-preserving case, and essentially  the same proof in King \cite{Ki86}  holds for infinite measure-preserving transformations.  

\begin{prop}
\label{propRigUncCen}
All rigid transformations have uncountable centralizer.
\end{prop}

These results suffice to establish the non-triviality of Theorem~\ref{thmMain}. Many common examples of infinite rank-one transformations are partially bounded, and hence not rigid. For example, the Hajian-Kakutani transformation defined in \cite{AdFrSi97}, is not rigid, giving another proof of Proposition 2.3 in  \cite{BaYa15}.

\begin{rmk}
Conditions (1) and (3) in the definition of partially bounded transformations, or close equivalents, are necessary in the proof of Theorem 5.1; without them, it is not difficult to construct examples which are rigid, and hence have uncountable centralizer. The authors are unaware of any transformation with non-trivial centralizer which satisfies Conditions (1) and (3) but not Condition (2).
\end{rmk}


\section{Preliminaries}\label{S:prelim}

In this Section and Section~\ref{S:mainthm} we shall work in the following setup: we have a partially bounded rank-one transformation \(T\) acting on an infinite, \(\sigma\)-finite Lebesgue space \(X\), and a measure-preserving transformation \(S\) such that \(S\circ T=T\circ S\) a.e. The transformation \(T\) comes equipped with parameters \((r_n)\), \((s_n)\)---where  we assume $s_n(r_n-1)=0$---and rank-one words \(w_n\). Moreover, by Lemma \ref{lemParBouUni}, we need not distinguish between a point \(x\in X\), and the \(T\)--\(P\)   name of \(x\), \(\Theta(x)\). We fix $\kappa \in \N$ so that \[|w_\kappa|>\mathfrak S.\] 
We shall take $n > \kappa$ to be fixed (what $n$ is will be determined later). We extend arguments of del Junco \cite{dJ79} and 
Gao and Hill \cite{GaHi16} to partially bounded transformations  to show that $S$ must be a power of $T$.

The following notions will be the key to translating information about when \(S(x)\) looks like a translate of \(x\) locally to information about when the same is true globally.



\begin{defn} \label{good}
An occurrence of $w_n$ at $i$ in $x$ is called \emph{good} if there exists an occurrence of $w_\kappa$ at $i$ in $S(x)$; otherwise the occurrence is called \emph{bad}.
\end{defn}

If a copy of $w_n$ at $i$ in $x$ is good, then there is a unique copy of $w_n$ in $S(x)$ that contains the $w_\kappa$ at position $i$. We shall write the beginning position of this $w_n$ as $i-\rho_{x,i}$, where $0 \leq \rho_{x,i} \leq |w_n| - |w_\kappa|$ (see Figure \ref{goodimage}). The following proposition is the crucial element in many of the arguments to come (see Figure \ref{a=bimage}).

	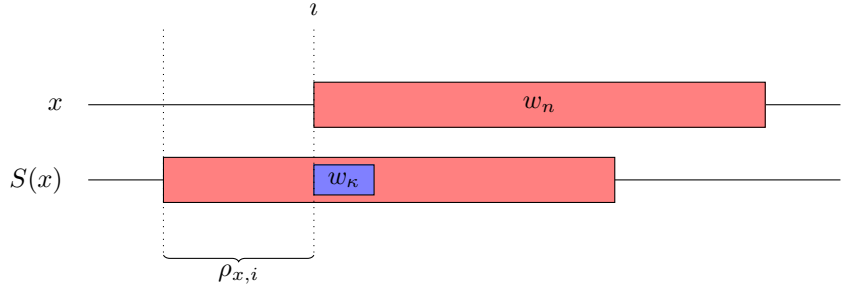
\begin{figure}
		\begin{tikzpicture}
			\node[left=6pt] at (0,1) {\(x\)};
			\node[left=6pt] at (0,0) {\(S(x)\)};
			\draw (0,1) -- (10,1);
			\draw (0,0) -- (10,0);
			\draw[fill=white!50!red] (3,0.7) rectangle (9,1.3) node[pos=.5] {\(w_n\)};
			\draw[fill=white!50!red] (1,-0.3) rectangle (7, 0.3);
			\draw[fill=blue!50!white] (3,-0.2) rectangle (3.8,0.2) node[pos=.5] {\(w_\kappa\)};
			\draw[dotted] (3,2) node[above=2pt] {\(i\)} -- (3,-1);
			\draw[dotted] (1,2) -- (1,-1);
			\draw[decorate,decoration={brace,mirror},below=3pt] (1,-1) -- (3,-1) node[pos=.5] {\(\rho_{x,i}\)};
			
		\end{tikzpicture}
		\caption{An illustration of a good occurrence of $w_n$ in $x$.}
		\label{goodimage}
	\end{figure}





\begin{prop} \label{a=b}
Let $x \in X$. Suppose that there is a good occurrence of $w_n$ at $i$ in $x$. Consider the string $w_n 1^a w_n$ that occurs at $i$ in $x$ and the string $w_n 1^b w_n$ that occurs at $i-\rho_{x,i}$ in $S(x)$. Then the next copy of $w_n$ in $x$ occurring at $i + |w_n| + a$ is good if and only if $a=b$.
\end{prop}
\begin{proof}
If $a=b$, then clearly the next copy of $w_n$ in $x$ will be good. Suppose that $a \ne b$. Write $a = s_{m_1}(\ell_1)$ and $b = s_{m_2}(\ell_2)$ for some $m_1, m_2 \geq n$ and \(\ell_1\in\lb0,r_{m_1}-1\rb\), $\ell_2\in\lb0,r_{m_2}-1\rb\). We consider the various cases:
\begin{case}
	Suppose $m_1 = m_2$. Then by Condition (2) in Definition \ref{defParBou}, $0 < |a - b| < \mathfrak{S}$. Since there is an occurrence of $w_\kappa$ at $i+|w_n|+b$ in $S(x)$ and $|w_\kappa| > \mathfrak S$, the copy of $w_n$ at $i + |w_n| + a$ in $x$ must be bad.
\end{case}
\begin{case}
	Suppose $m_1 < m_2$. Then by Condition (3), there must be a 1 at $i+|w_n|+a$ in $S(x)$ since $s_{m_2}(\ell_2) > s_{m_1}(\ell_1) + \rho_{x,i}$. Hence the copy of $w_n$ at $i+|w_n|+a$ in $x$ is bad.
\end{case}
\begin{case}
	Suppose $m_1 > m_2$. Consider the copy of $w_{m_1}$ that contains the copy of $w_n$ at $i-\rho_{x,i}$ in $S(x)$. By Condition (3), this copy must end at some $j < i + |w_n| + a$. Then there is a long stretch of 1's of length $s_{m_1}(\ell_3)$ that begins at $j+1$ in $S(x)$ and which, by Condition (2), will end after position $i+|w_n|+a$. It follows that the occurrence of $w_n$ at $i+|w_n|+a$ in $x$ is bad.
\end{case}
\end{proof}


\begin{figure}
\begin{tikzpicture}

\node[left=6pt] at (0,1) {$x$};
\node[left=6pt] at (0,0) {$S(x)$};
\draw (0,1) -- (12,1);
\draw (0,0) -- (12,0);
\draw[fill=white!50!red] (2,0.7) rectangle (4,1.3) node[pos=.5] {$w_n$};
\draw[fill=white!50!red] (9,0.7) rectangle (11,1.3) node[pos=.5] {$w_n$};
\draw[fill=white!50!red] (1.5,-0.3) rectangle (3.5,0.3);
\draw[fill=white!50!red] (7.5,-0.3) rectangle (9.5,0.3) node[pos=.5] {$w_n$};
\draw[fill=black!10!white] (4.1,0.75) rectangle (8.9,1.2) node[pos=.5] {$1^a$};
\draw[fill=black!10!white] (3.6,-0.25) rectangle (7.4,0.2) node[pos=.5] {$1^b$};

\draw[fill=blue!50!white] (2,-0.2) rectangle (2.5,0.2) node[pos=.5] {$w_\kappa$};
\draw[dotted] (2,2) node[above=2pt] {\(i\)} -- (2,-1);
\draw[dotted] (1.5,2) -- (1.5,-1);
\draw[decorate,decoration={brace,mirror},below=3pt] (1.5,-1) -- (2,-1) node[pos=.5] {\(\rho_{x,i}\)};

\end{tikzpicture}
\caption{An illustration of the situation in Proposition \ref{a=b}: if the left $w_n$ is good, then the next $w_n$ is good if and only if $a = b$.}
\label{a=bimage}
\end{figure}
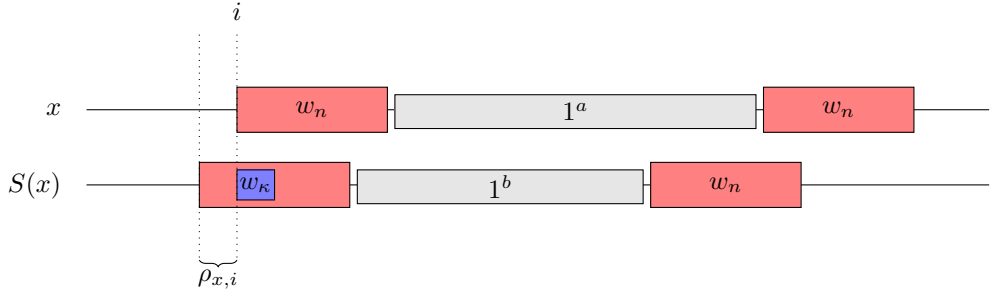

\begin{rmk} \label{remark}
An analogous argument proves that if a copy of $w_n$ is good, then the \lq\lq previous" copy of $w_n$ is good if and only if $a=b$ (where now $a$ and $b$ denote the lengths of the stretches of 1's to the left).
\end{rmk}

The following  lemma gives us a condition we use to show that  \(S\) must be a power of \(T\).

\begin{lem}
\label{goodlemma} Fix  $x\in X$ and  $n\in\N$ with $n>k$. 
If every copy of \(w_n\) in \(x\) is good, then there is $\ell\in\N\cup\{0\}$ such that
		\[S(x)=T^{\ell}(x).\]
\end{lem}

\begin{proof}
Start with some good occurrence of $w_n$ at some $i$ in $x$. Then there is an occurrence of $w_n$ at $i-\rho_{x,i}$ in $S(x)$. Consider the next copy of $w_n$ in $x$ (at $i + |w_n| + a$ in the notation used in Proposition \ref{a=b}). Since it is good, Proposition \ref{a=b} implies that the next copy of $w_n$ in $S(x)$ occurs at $i-\rho_{x,i}+|w_n|$. We can repeat this argument for all the copies of $w_n$ occurring to the right of the $w_n$ we started out with. Moreover, by Remark \ref{remark}, we can do the same for the copies of $w_n$ occurring to the left. This proves that $S(x) = T^{\rho_{x,i}}(x)$ and we can set $\ell=\rho_{x,i}$. Here we use the fact that there is no first or last occurrence of $w_n$ in $x$.
\end{proof}

\section{Proof of the Main Theorem}\label{S:mainthm}

We are now equipped to prove our main result.

\begin{thm}
\label{thmMain}
If \(T\) is a partially bounded rank-one transformation, any transformation \(S' \in C(T)\) is a power of \(T\).
\end{thm}

For \(m>n\), we shall say that a copy of $w_m$ is \emph{totally good} if every $w_n$ that occurs in it is good; we also say that a copy of $w_m$ is \emph{totally bad} if every $w_n$ that occurs in it is bad. Suppose that $x$ has a totally good occurrence of $w_m$ at $i$. Then $S(x)$ has an occurrence of $w_n$ at $i-\rho_{x,i}$. Moreover, by repeated application of Proposition \ref{a=b}, $S(x)$ has an occurrence of $w_m$ at $i-\rho_{x,i}$ (see Figure \ref{totallygoodimage}). The following lemma is the analogue of Proposition \ref{a=b} for totally good occurrences.

	
	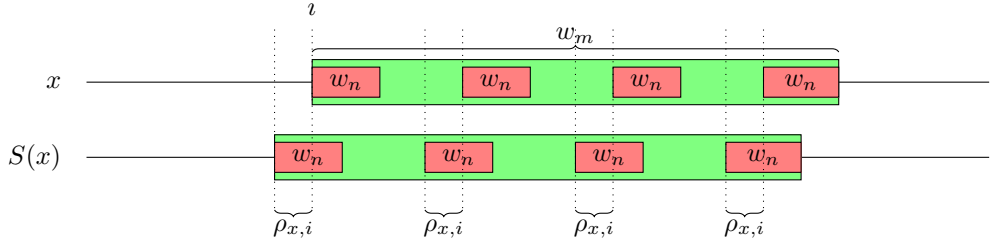
\begin{figure}
			\begin{tikzpicture}
				\node[left=6pt] at (0,1) {\(x\)};
				\node[left=6pt] at (0,0) {\(S(x)\)};
				\draw (0,1) -- (12,1);
				\draw (0,0) -- (12,0);
				\draw[fill=green!50!white] (3,.7) rectangle (10,1.3);
				\draw[fill=green!50!white] (2.5,-.3) rectangle (9.5,.3);
				\draw[fill=white!50!red] (3,.8) rectangle (3.9,1.2) node[pos=.5] {\(w_n\)};
				\draw[fill=white!50!red] (5,.8) rectangle (5.9,1.2) node[pos=.5] {\(w_n\)};
				\draw[fill=white!50!red] (7,.8) rectangle (7.9,1.2) node[pos=.5] {\(w_n\)};
				\draw[fill=white!50!red] (9,.8) rectangle (10,1.2) node[pos=.5] {\(w_n\)};
				\draw[fill=white!50!red] (2.5,-.2) rectangle (3.4,.2) node[pos=.5] {\(w_n\)};
				\draw[fill=white!50!red] (4.5,-.2) rectangle (5.4,.2) node[pos=.5] {\(w_n\)};
				\draw[fill=white!50!red] (6.5,-.2) rectangle (7.4,.2) node[pos=.5] {\(w_n\)};
				\draw[fill=white!50!red] (8.5,-.2) rectangle (9.5,.2) node[pos=.5] {\(w_n\)};
				\draw[dotted] (3,1.7) node[above=2pt] {\(i\)}  -- (3,-.7);
				\draw[dotted] (2.5,1.7) -- (2.5,-.7);
				\draw[decorate,decoration={brace,mirror},below=3pt] (2.5,-.7) -- (3,-.7) node[pos=.5] (a) {\(\rho_{x,i}\)};
				\draw[decorate,decoration={brace,mirror},below=3pt] (4.5,-.7) -- (5,-.7) node[pos=.5] (a) {\(\rho_{x,i}\)};
				\draw[decorate,decoration={brace,mirror},below=3pt] (6.5,-.7) -- (7,-.7) node[pos=.5] (a) {\(\rho_{x,i}\)};
				\draw[decorate,decoration={brace,mirror},below=3pt] (8.5,-.7) -- (9,-.7) node[pos=.5] (a) {\(\rho_{x,i}\)};
				\draw[dotted] (4.5,1.7) -- (4.5,-.7);
				\draw[dotted] (5,1.7) -- (5,-.7);
				\draw[dotted] (6.5,1.7) -- (6.5,-.7);
				\draw[dotted] (7,1.7) -- (7,-.7);
				\draw[dotted] (8.5,1.7) -- (8.5,-.7);
				\draw[dotted] (9,1.7) -- (9,-.7);
				
				\draw[decorate,decoration=brace,above=3pt] (3, 1.4) -- (10,1.4) node[pos=.5] {\(w_m\)};
			\end{tikzpicture}
			\caption{An illustration of a totally good occurrence of $w_m$ in $x$.}
			\label{totallygoodimage}
		\end{figure}

The next lemma is the analogue of Proposition~\ref{a=b} for totally good occurrences, and its proof is very similar.
\begin{lem}
\label{new a = b}
Let $x \in X$. Suppose that there is a totally good occurrence of $w_m$ at $i$ in $x$. Write $w_m 1^a w_m$ for the string beginning at $i$ in $x$ and $w_m 1^b w_m$ for the string beginning at $i - \rho_{x,i}$ in $S(x)$. Then the ``next" $w_m$ occurring at $i + |w_m| + a$ in $x$ is totally good if $a=b$, and is totally bad if $a \neq b$.
\end{lem}
\begin{proof}
\setcounter{case}{0}
If $a = b$, then clearly the next $w_m$ will be totally good. Now suppose that $a \not = b$. Write $a = s_{m_1}(\ell_1)$ and \(\ell_1\in\lb0,r_{m_1}-1\rb\), $\ell_2\in\lb0,r_{m_2}-1\rb\). We distinguish three cases:
\begin{case}Suppose that $m_1 = m_2$. By Condition (2) of Definition \ref{defParBou}, $0 < |a - b| < \mathfrak{S}$. Since $|w_\kappa| > \mathfrak{S}$, all the $w_n$ composing the $w_m$ at $i+|w_m|+a$ in $x$ will be bad.
\end{case}
\begin{case}
Suppose that $m_1 < m_2$. Then by Condition (3) of Definition \ref{defParBou}, there is a long stretch of 1's in $S(x)$ that begins at $i - \rho_{x,i} + |w_m|$ so that all the $w_n$'s composing the $w_m$ at $i+|w_m|+a$ in $x$ will be bad.
\end{case}
\begin{case}
Suppose that $m_1 > m_2$. Consider the copy of $w_{m_1}$ that contains the copy of $w_n$ at $i - \rho_{x,i}$ in $S(x)$. By Condition (3), this copy must end at some $j < i + |w_m| + a$. Then there is a long stretch of 1's of length $s_{m_1}(\ell_3)$ that begins at $j+1$ in $S(x)$ and which, by conditions (2) and (3), will cause all the $w_n$ composing the following copy of $w_m$ to be bad.
\end{case}
\end{proof}

\begin{proof}[Proof of Theorem~\ref{thmMain}]
Using that the sets
	\begin{equation}
	\label{fineq}
		E_{w_n,i} = \{x \in X: x \text{ has an occurrence of $w_n$ at $i$}\}=T^i(B_n)
	\end{equation}
are dense in the measure algebra, we can find $n > k$ such that
	\[ \frac{\mu[ E_{w_n,i} \cap (S')^{-1}(E_{w_\kappa,0}) ]}{\mu[ E_{w_n,i}]} > 1 - \frac{1}{2\mathfrak R+1} .\]
Let $S = S' \circ T^{-i}$. Then $S$ commutes with $T$, and $S$ is a power of $T$ if and only if $S'$ is a power of $T$. Moreover, by Equation \ref{fineq},
\[ E_{w_n,i} \cap (S')^{-1}(E_{w_\kappa,0}) = T^{-i} (E_{w_n,0} \cap S^{-1}(E_{w_\kappa,0})).\]
Thus,
\[ \frac{\mu[ E_{w_n,0} \cap S^{-1}(E_{w_\kappa,0}) ]}{\mu[ E_{w_n,i}]} > 1 - \frac{1}{2\mathfrak R+1}.\]
Note that $T^i(x) \in E_{w_n,0} \cap S^{-1}(E_{w_\kappa,0})$ exactly when $x$ has a good copy of $w_n$ at $i$ (with respect to S). Thus by the Hopf ratio ergodic theorem, for a.e. $x \in X$,
	\begin{equation}
	\label{eq1}
		\lim_{N \to \infty} \frac{ \text{$\#|$\{good occurrences of $w_n$ in $x\lb0,N\rb\}|$}}{ \text{$\#|$\{occurrences of $w_n$ in $x\lb0,N\rb\}|$}} > 1 - \frac{1}{2\mathfrak R+1}.
	\end{equation}

We prove that there cannot be a bad occurrence of $w_n$ for a.e. $x$. If there is a bad occurrence of $w_n$ in $x$, then there must exist a stretch of $d$ consecutive bad occurrences of $w_n$, preceded by at least $2\mathfrak R\cdot d$ good occurrences of $w_n$ and followed by at least one good occurrence of $w_n$. Otherwise, Inequality \ref{eq1} would fail. (See Figure \ref{situation}.) We need to show that such a situation is impossible.

Choose $m > n$ so that
	\[\#|\{\text{copies of $w_n$ in $w_{m-1}$}\}|\leq d <\#|\{\text{copies of $w_n$ in $w_m$}\}|\]
Thus the number of $w_n$ in $w_m$ is at most $\mathfrak R\cdot d$. It follows that there is an entire copy of $w_m$ contained in the stretch of $2\mathfrak R\cdot d$ good occurrences of $w_n$. Thus, this $w_m$ must be totally good. By Lemma \ref{new a = b}, the next $w_m$ must be either totally good or totally bad. If it is good, then the next $w_m$ must be either totally good or totally bad, and so on. We shall only reach a totally bad copy of $w_m$ when we get to the stretch of $d$ bad copies of $w_n$. However, since \(d\) is less than the number of copies of $w_n$ in $w_m$, a totally bad copy of $w_m$ would require more than $d$ bad copies of $w_n$. This yields a contradiction.

By Lemma \ref{goodlemma}, we obtain that for a.e. $x$, $S(x) = T^\ell(x)$ for some $\ell$ possibly depending on $x$. We conclude the proof by noting that for each \(\ell \in\B Z\), the sets \(A_\ell=\{x\in X: S(x)=T^{\ell}(x)\}\) are invariant under \(T\) by commutativity of \(S\) and \(T\), and \(X = \bigcup_{\ell \in\B Z} A_n\). So by ergodicity, one of them has null complement. Therefore, \(S=T^\rho\) a.e.~for some \(\rho\). 
\end{proof}

 \begin{figure}
	\begin{tikzpicture}
			\draw[left color=white, right color=red!70!white,draw=none] (-.4,-.2) rectangle (0,.2);
			\draw[fill=red!70!white,draw=none] (0,-.2) rectangle (1,.2) node[pos=.5] {bad};
			\draw[fill=green!70!white,draw=none] (1,-.2) rectangle (6,.2) node[pos=.5] {good};
			\draw[fill=red!70!white,draw=none] (6,-.2) rectangle (7,.2) node[pos=.5] {bad};
			\draw[fill=green!70!white,draw=none] (7,-.2) rectangle (9,.2) node[pos=.5] {good};
			\draw[fill=red!70!white,draw=none] (9,-.2) rectangle (10,.2) node[pos=.5] {bad};
			\draw[right color=white, left color=red!70!white,draw=none] (10,-.2) rectangle (10.4,.2);
			\draw[decorate,decoration={brace},above=3pt] (6,.2) -- (7,.2) node[pos=.5] {\(d\)-many copies};
			\draw[decorate,decoration={brace,mirror},below=3pt] (1,-.2) -- (6,-.2) node[pos=.5] {\(\geq 2 \mathfrak{R} \cdot d\)-many copies};
		\end{tikzpicture}
		\caption{An illustration of the stretch of copies considered in the proof of Theorem \ref{thmMain}.}
\label{situation}
\end{figure}
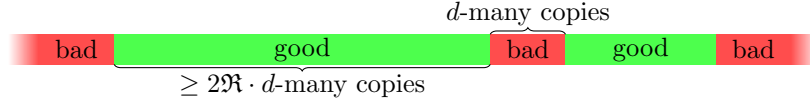

\begin{rmk} We note that when $S$ is measure-preserving, our proof does not require $S$ to 
be invertible. Also, when $S$ is nonsingular and invertible, by \cite{BSSSW15}, since \(T\) is rank-one, $S$ must be  measure preserving.
\end{rmk}

\begin{rmk} 
We have also  extended our arguments to show that a  class of infinite rank-one flows, which we call \emph{partially bounded}, have trivial centralizers. In this context, the \(T\)--\(P\) names of points are maps from $\R$ into $\{0,1\}$ instead of being maps from $\Z$ into $\{0,1\}$. The details will appear in a forthcoming paper, along with further generalizations to partially bounded $\Z^d$-actions and canonically bounded flows.
\end{rmk}

\section{Partially bounded transformations isomorphic to their inverse}\label{S:isoinv}

In this section we extend the methods of \cite{Hi16} to characterize when a partially bounded transformation is isomorphic to its inverse. Since we have already assumed that \(s_n(r_n-1)=0\), from this point on, we shall adopt the convention that the spacer parameter $(s_n)$ is a sequence of $(r_n-1)$-tuples, instead of $r_n$-tuples. Thus, $s_n = (s_n(0), s_n(1), \dots, s_n(r_n-2))$. We prove the following theorem.  

\begin{thm} \label{theorem}
Let $T$ be a partially bounded rank-one transformation with parameters $(r_n)$ and $(s_n)$. Then $T$ is isomorphic to $T^{-1}$ if and only if there exists $N \in \N$ such that for all $n \geq N$, $s_n = \overline{s_n}$ (where $\overline{s_n}$ denotes the reverse of $s_n$, i.e. $s_n(i) = s_n(r_n-2-i)$).
\end{thm}

\begin{defn} \label{perp}
Let $s$ and $s'$ be two finite sequences of integers of the same length. We say that $s$ and $s'$ are \emph{incompatible} if there does not exist an integer $c$ such that $s$ is a substring of $s' c s'$. This is a symmetric relation. 
\end{defn}

\begin{prop} \label{incompatible}
Suppose $(X, \mu, T)$ and $(Y, \nu, S)$ are two rank-one systems with respective parameters $(r_n), (s_n)$ and $(r_n), (s_n')$ (with no spacers added to the last subcolumn). Let $(v_n)$ and $(w_n)$ be the corresponding sequences of rank-one words. If for some $n$, $s_n$ and $ s_n'$ are incompatible, then the string
\[ w_n 1^{s_n(0)} w_n 1^{s_n(1)} \cdots 1^{s_n(r_n-2)} w_n \]
cannot occur in any $y \in Y$.
\end{prop}

The next proposition is a direct generalization of Proposition 2.1 in \cite{Hi16}; it provides sufficient conditions on pairs of rank-one transformations to guarantee non-isomorphism.

\begin{prop} \label{prop}
Let $(X, \mu, T)$ and $(Y, \nu, S)$ be two rank-one systems, and let $(r_n), (s_n)$ and $(r_n'), (s_n')$ be their respective cutting and spacer parameters. Write $(v_n)$ and $(w_n)$ for their respective sequences of rank-one words. Suppose the following hold:
\begin{enumerate}
\item For all $n$, $r_n = r_n'$ and $\sum_{i=0}^{r_n-2} s_n(i) = \sum_{i=0}^{r_n'-2} s_n'(i)$. In particular, $|v_n| = |w_n|$ for all $n$. (We say that the parameters are \emph{commensurable}.)
\item There is $\mathfrak S>0$ such that for all $n$ and $0 \leq i,j< r_n-1$, $|s_n(i)-s_n'(j)| <\mathfrak S$.
\item For all $n$ and $0 \leq  i,j < r_n-1$, $s_n(i) \geq |v_n|$ and $s_n'(i) \geq |w_n|$.
\item There exists a subsequence of rank-one words $(v_{n_\ell})$ and $(w_{n_\ell})$ with corresponding cutting and spacer parameters $(q_\ell)$, $(t_\ell)$ and $(q_\ell)$, $(t_\ell')$, respectively, such that for some $\mathfrak{Q} > 0$ and infinite set $\mathbb{M} \subset \N$ we have for all $\ell \in \mathbb{M}$, \(q_\ell < \mathfrak{Q}\) and $t_\ell$ and $t_\ell'$ are incompatible.
\end{enumerate}
Then $X$ and $Y$ are not isomorphic.
\end{prop}
\begin{proof}
Assume for contradiction that there exists an isomorphism $\phi: X \to Y$ so that $\phi \circ T = S \circ \phi$. Choose $k$ such that $|v_\kappa| = |w_\kappa| > \mathfrak S$. Then we can find $n = n_\ell$ with $\ell \in \mathbb{M}$, $n > k$, and $j \in \Z$ such that
	\begin{equation}
	\label{limit}
		\frac{ \mu[ E_{v_n,j} \cap \phi^{-1}(E_{w_\kappa, 0})]}{\mu[E_{v_n,j}]} > 1 - \frac{1}{\mathfrak R}.
	\end{equation}
Let $k, n \in \mathbb{N}$ and $j \in \Z$ be as above. Similarly to what we have defined before, say that an occurrence of $v_n$ at $i \in \Z$ in $x \in X$ is \emph{good} if there exists an occurrence of $w_\kappa$ at $i-j$ in $\phi(x)$; otherwise $v_n$ is called \emph{bad}. Note that this copy of $w_\kappa$ in $\phi(x)$ must be contained in a (unique) copy of $w_n$ beginning at $i-j-\rho_{x,i}$ for some $0 \leq \rho_{x,i} \leq |w_n| - |w_\kappa|$. We have that \(x\) has a good occurrence of $v_n$ at $j$ if and only if \(x \in E_{v_n,j} \cap \phi^{-1}(E_{w_\kappa, 0}).\)

Moreover, for each $i \in \Z$, $x$ has a good occurrence of $v_n$ at $i+j$ if and only if \(T^i(x) \in E_{v_n,j} \cap \phi^{-1}(E_{w_\kappa, 0})\). By Hopf's Ratio Ergodic Theorem, we conclude that for a.e. $x \in X$,
	\[\lim_{N \to \infty} \frac{ \#|\{\text{good occurrences of $v_n$ in $x\lb0,N\rb$\}}|}{\#|\{\text{occurrences of $v_n$ in $x\lb0,N\rb$}\}|} > 1 - \frac{1}{\mathfrak R}.\]

Say that an occurrence of $v_m = v_{n_{\ell+1}}$ at $i \in \Z$ in $x \in X$ is \emph{totally good} if all the $v_n = v_{n_\ell}$ composing it are good. By Inequality \ref{limit}, almost every $x \in X$ must contain a totally good occurrence of $v_m$.

\begin{clm}
	Suppose the string $v_n 1^a v_n$ occurs at $i \in \Z$ in $x \in X$ and that the first $v_n$ at $i$ is good. We have a string of the form $w_n 1^b w_n$ occurring at $i-j-\rho_{x,i}$ in $\phi(x)$. Then the next copy of $v_n$ at $i+|v_n|+a$ is good if and only if $a = b$.
\end{clm}

The proof of the claim is similar to the proof of Proposition~\ref{a=b} and we leave it to the reader.

Let $x \in X$ contain a totally good occurrence of \[v_m = v_n 1^{t_\ell(0)} v_n 1^{t_\ell(1)} \cdots 1^{t_\ell(q_\ell-2)} v_n.\] By repeatedly applying the claim, we conclude that there must be a string of the form
\[ w_n 1^{t_\ell(0)} w_n 1^{t_\ell(1)} \cdots 1^{t_\ell(q_\ell-2)} w_n \]
in $\phi(x)$. But this contradicts the assumption that $t_\ell$ and $t_\ell'$  are incompatible, by Proposition \ref{incompatible}.
\end{proof}


We shall need the following lemma, which appears in \cite{Hi16} as Lemma 2.2. Given two finite sequences $s_1 = (s_1(0), \dots, s_1(r_1-2))$ and $s_2 = (s_2(0), \dots, s_2(r_2-2))$, we define $s_2 * s_1$ to be the following sequence of length $r_1 r_2 - 1$:
\[ s_2 * s_1 := s_1 s_2(0) s_1 s_2(1) \cdots s_2(r_2-2) s_1 .\]
Note that $s_2 * s_1$ is just the spacer parameter term obtained by replacing two consecutive stages in the construction of a rank-one transformation by a single one. It is easy to check that $*$ is an associative operation.

\begin{lem} \label{lemma}
Let $s_1 = (s_1(0), \dots, s_1(r_1-2))$ and $s_2 = (s_2(0), \dots, s_2(r_2-2))$. Suppose that $s_1 \ne \overline{s_1}$ and that $s_2$ is not constant. Then $s_2 * s_1$ and $ \overline{s_2 * s_1}$ are incompatible.
\end{lem}

We are now ready to prove Theorem \ref{theorem}. The backwards direction is easy: observe that for any rank-one system $(X, \mu, T)$ with parameters $(r_n)$ and $(s_n)$, $(X, \mu, T^{-1})$ is isomorphic to the rank-one system $(\overline{X}, \overline{\mu}, T)$ with parameters $(r_n)$ and $(\overline{s_n})$. Thus if the condition in Theorem \ref{theorem} holds, then a possible isomorphism $\phi$ between $(X, \mu, T)$ and $(\overline{X}, \overline{\mu}, T)$ is the one where $\phi(x)$ is obtained by replacing every expected occurrence of $v_N$ in $x$ by an occurrence of $v_N'$. This is an example of the \emph{stable isomorphisms} discussed in \cite{GaHi16}, and is in particular a topological isomorphism (which also gives an  isomorphism of the Borel systems). Clearly, it is also a finitary isomorphism.

We establish the forward direction of the theorem below.

\begin{proof}[Proof of Theorem \ref{theorem}]
Let $(X, \mu, T)$ be a partially bounded rank-one system with parameters $(r_n)$ and $(s_n)$ and associated rank-one words $(v_n)$. Suppose that $s_n \ne \overline{s_n}$ for infinitely many $n$. Let $(Y, \nu,T)$ be the rank-one system with parameters $(r_n)$ and $(\overline{s_n})$, which is isomorphic to the inverse of $(X, \mu, T)$. We wish to apply Proposition \ref{prop} to show that $X$ and $Y$ are not isomorphic. The only nontrivial condition to check is the existence of a subsequence of $(u_\ell) = (v_{n_\ell})$ of $(v_n)$ satisfying the criteria in condition (4), which we do below.

Let $u_0 = v_0 = 0$. Suppose that $u_{2m}$ has been defined as $v_k$. Let $n > k$ such that $s_n \ne \overline{s_n}$. Define $u_{2m+1} = v_n$ and $u_{2m+2} = v_{n+3}$. In this way, we have obtained a new sequence of rank-one words $(u_\ell)$ of $(X, \mu,T)$; let $(q_\ell)$ and $(t_\ell)$ be the new cutting and spacer parameters. We check that there exists $\mathfrak{Q} > 0$ and an infinite set $\mathbb{M} \subset \N$ such that $q_\ell < \mathfrak{Q}$ and $t_\ell$ and $ \overline{t_\ell}$ are incompatible for all $\ell \in \mathbb{M}$. Let $\mathbb{M}$ be the set of odd positive integers. Then for all $\ell \in \mathbb{M}$ we have:
\begin{itemize}
\item $q_\ell = r_{n+2} \cdot r_{n+1} \cdot r_n < \mathfrak R^3$,
\item $t_\ell = (s_{n+2} * s_{n+1}) * s_n$.
\end{itemize}
We can take $\mathfrak{Q} = \mathfrak R^3$. Moreover, using condition (3) in the definition of partial boundedness, we see that $s_{n+2} * s_{n+1}$ is not constant. Since $s_n \ne \overline{s_n}$, Lemma \ref{lemma} implies that \((s_{n+2} * s_{n+1}) * s_n\) and \(\overline{(s_{n+2} * s_{n+1}) * s_n}\) are incompatible. In other words, $t_\ell$ and  $ \overline{t_\ell}$ are incompatible, which concludes the proof.
\end{proof}

\bibliographystyle{plain}
\bibliography{ErgodicBibMaster-Cen-Apr29}

\end{document}